\documentclass{amsart}
\usepackage{amsthm,amsmath, amsfonts, amssymb, comment}
\newtheorem{theorem}{Theorem}[section]
\newtheorem{lemma}[theorem]{Lemma}
\newtheorem{proposition}[theorem]{Proposition}
\newtheorem{corollary}[theorem]{Corollary}
\newtheorem{question}[theorem]{Question}

\theoremstyle{definition}
\newtheorem{definition}{Definition}[section]
\newtheorem*{programme}{Program}
\theoremstyle{remark}

\usepackage{mathrsfs}
\usepackage[english]{babel}
\usepackage{graphicx}   
\usepackage{tikz}
\usepackage{subcaption}
\usepackage{url}        
\usepackage{bm}         
\usepackage{multirow}
\usepackage{booktabs}
\usepackage{algorithm}
\usepackage{algorithmic}
\usepackage{enumerate}

\usepackage[colorlinks, linkcolor=black]{hyperref}

\providecommand*{\remarkautorefname{Remark}}
\hypersetup{colorlinks, bookmarks, unicode} 
\usepackage{multicol}
\usepackage[noadjust]{cite}
\usepackage{indentfirst}

\begin{document}

\title[Extremely Amenable Automorphism Groups]{Extremely Amenable Automorphism Groups of Countable Structures}


\author{Su Gao}
\address{School of Mathematical Sciences and LPMC, Nankai University, Tianjin 300071, P.R. China}
\email{sgao@nankai.edu.cn}
\thanks{S. Gao and the F. Li acknowledge the partial support of their research by the National Natural Science Foundation of China (NSFC) grants 12271263 and 12250710128. R. Li acknowledges the partial support of his research by the NSFC grant 124B2001. The authors attended the Tianyuan Workshop on Computability Theory and Descriptive Set Theory in June 2025 when a part of this research was conducted, and we acknowledge the support by the Tianyuan Mathematics Research Center.}
\author{Feng Li}
\address{School of Mathematical Sciences and LPMC, Nankai University, Tianjin 300071, P.R. China}
\email{fengli@mail.nankai.edu.cn}

\author{Ruiwen Li}
\address{School of Mathematical Sciences and LPMC, Nankai University, Tianjin 300071, P.R. China}
\email{rwli@mail.nankai.edu.cn}

\date{\today}

\begin{abstract}  In this paper we study the isomorphism relation of extremely amenable non-Archimedean Polish groups from the point of view of descriptive set theory.  Our main result is that this relation is more complex than any isomorphism relation of countable structures in the Borel reducibility hierarchy. 
\end{abstract}

\maketitle
\section{Introduction}    

A topological group $G$ is extremely amenable if for every continuous action of $G$ on a compact Hausdorff space $X$, there is a fixed point $x\in X$, i.e., for all $g\in G$, $g\cdot x=x$. This is a remarkable property for a topological group to have, but many automorphism groups of countable structures have been shown to be extremely amenable since the late 1990s. Pestov \cite{Pestov98} gave the first example of such a group, which is ${\rm Aut}(\mathbb{Q}, <)$, the automorphism group of a countable dense linear order without endpoints. Motivated by the idea in \cite{Pestov98}, Kechris--Pestov--Todorcevic \cite{KPT} gave a complete characterization of when the automorphism group of a countable structure is extemely amenable; this very influential theorem is now called the KPT correspondence. The KPT correspondence connects the extreme amenability  of automorphism groups to the so-called structural Ramsey property of Fra\"iss\'e classes, and gives a number of examples of extremely amenable groups.  Through the work of Ne\v{s}et\v{r}il--R\"{o}dl \cite{NR77} \cite{NR83} and Ne\v{s}et\v{r}il \cite{Nes07} on Ramsey classes, we now have many extremely amenable groups of the form $\mbox{\rm Aut}(M)$, where $M$ is the Fra\"{\i}ss\'{e} limit of a Fra\"{\i}ss\'{e} class of finite structures. Such structures include ordered graphs, ordered $K_n$-free graphs, ordered $\mathbb{Q}$-metric spaces, and ordered $\mathbb{N}$-metric spaces, etc. The KPT correspondence has also been generalized to the context of continuous logic by Melleray--Tsankov \cite{MT}, thus giving examples of extremely amenable groups which are automorphism groups of uncountable structures.

Despite the diversity of these examples, not much was known about the number of distinct isomorphism types of extremely amenable groups. Among the above examples, only finitely many pairwise non-isomorphic extremely amenable groups can be identified. This motivates us to study the following question: How many pairwise non-isomorphic extremely amenable automorphism groups of countable structures are there? 

By a well-known theorem of Becker--Kechris \cite[Theorem 1.5.1]{BK}, the class of all automorphism groups of countable structures, up to topological group isomorphism, coincides exactly with the class of all non-Archimedean Polish groups, as well as with the class of all closed subgroups of the infinite permutation group $S_\infty$. Thus, our question can be equivalently formulated as: how many pairwise non-isomorphic extremely amenable non-Archimedean Polish groups are there?

Descriptive set theory provides a framework for answering such questions beyond counting cardinalities. From the point of view of descriptive set theory, we regard the isomorphism relation of all extremely amenable non-Archimedean Polish groups as an equivalence relation on a standard Borel space, and use the notion of Borel reducibility to compare it with benchmark equivalence relations. If $E, F$ are equivalence relations on standard Borel spaces $X, Y$, respectively, we say that $E$ is \emph{Borel reducible} to $F$, or $F$ \emph{Borel reduces} $E$, denoted $E\leq_B F$, if there is a Borel function $f\colon X\to Y$ such that, for all $x, x'\in X$,
$$ xEx'\iff f(x)Ff(x'). $$
The Borel reducibility notion gives rise to a theory of relative complexity among equivalence relations on standard Borel spaces. If one can show $=_{\mathbb{R}}\ \leq_B E$, where $=_{\mathbb{R}}$ is the equality relation on $\mathbb{R}$, then in particular one has established that there are continuum many pairwise $E$-inequivalent elements of $X$. For a treatise of this theory, the reader can consult \cite{GaoBook}.

The notion of Borel reducibility was introduced by Friedman--Stanley in \cite{FS}. In \cite{FS} the authors also introduced an important benchmark equivalence relation known as the \emph{graph isomorphism}, which is the isomorphism relation of all countable graphs, and showed that the isomorphism relation on any Borel class of countable structures is Borel reducible to the graph isomorphism. Moreover, the graph isomorphism is analytic and non-Borel, and Borel reduces $=_{\mathbb{R}}$. Our main theorem of this paper is the following.

\begin{theorem}\label{thm:main7}  The isomorphism relation of all extremely amenable non-Archimedean Polish groups Borel reduces the graph isomorphism.
\end{theorem}

This has the following immediate corollary literally answering our motivating question.

\begin{corollary}\label{thm:main6} There are continnum many pairwise non-isomorphic extremely amenable non-Archimedean Polish groups. Furthermore, the isomorphism relation of all extremely amenable non-Archimedean Polish groups is analytic and non-Borel.
\end{corollary}

Our Theorem~\ref{thm:main7} can be viewed also as a contribution to a research program initiated by Kechris--Nies--Tent \cite{KNT}, which calls for studying the complexity of the isomorphism relation on naturally defined Borel classes of non-Archimedean Polish groups.

The rest of the paper is organized as follows. In Section~\ref{sec:2} we define the basic concepts and notions needed for the rest of the paper. In particular, we verify that the class of all closed subgroups of $S_\infty$ which are extremely amenable is Borel. In  Section~\ref{sec:n3} we consider the class of all ordered $\Delta$-metric Urysohn spaces and characterize the isomorphism relation between the automorphism groups of their Fra\"{\i}se\'e limits. In Section~\ref{sec:n4} we show that the equivalence relation between countable distance value sets Borel reduces the graph isomorphism, and thus derive Theorem~\ref{thm:main7}.

\medskip
{\it Acknowledgment.} We thank Filippo Calderoni, Wei Dai, Mahmood Etedadialiabadi, Aleksander Iwanow, Andre Nies, Marcin Sabok, Assaf Shani, Simon Thomas, and V\'{\i}ctor Hugo Ya\~{n}ez for their inspiring comments on earlier versions of the paper. 

\medskip
{\it AI statement.} No generative AI was used in the research and the presentation of this paper.

\section{Preliminaries\label{sec:2}}
In this section we review some of the basic notions and results we will use in the rest of the paper.

\subsection{Fra\"iss\'e Theory} Our standard reference for model theory and Fra\"iss\'e theory is \cite{Hodges}.

A {\em language} $\mathcal{L}$ is a set consisting of the following:
\begin{enumerate}
\item[(a)] a collection $(R_i)_{i\in I}$ of relation symbols; for each $i \in I$, there is an associated arity $n(i)$, which is a positive integer, such that $R_i$ is an $n(i)$-ary relation symbol;
\item[(b)] a collection $(F_j)_{j\in J}$ of function symbols; for each $j \in J$, there is an associated arity $m(j)$, which is a natural number, such that $F_j$ is an $m(j)$-ary function symbol.
\end{enumerate}
The $0$-ary function symbols, if any, are also called constant symbols. We say $\mathcal{L}$ is {\em countable} if $I$ and $J$ are countable.

If $\mathcal{L}=((R_i)_{i\in I}, (F_j)_{j\in J})$ is a language, an {\em $\mathcal{L}$-structure} is a tuple
$$\mathcal{A}= (A; (R_i^{\mathcal{A}})_{i\in I}, (F_j^{\mathcal{A}})_{j \in J} )$$ where 
\begin{enumerate}
\item[(i)] $A$ is a set known as the {\em universe} of $\mathcal{A}$ and denoted as $|\mathcal{A}|$;
\item[(ii)] for each $i\in I$, $R_i^{\mathcal{A}}$ is the {\em interpretation} of $R_i$ in $\mathcal{A}$, which is an $n(i)$-ary relation on $A$, i.e., $R_i^{\mathcal{A}}\subseteq A^{n(i)}$;
\item[(iii)] for each $j\in J$, $F_j^{\mathcal{A}}$ is the {\em interpretation} of $F_j$ in $\mathcal{A}$, which is an $m(j)$-ary function from $A^{m(j)}$ to $A$, i.e., $F_j^{\mathcal{A}}\colon A^{m(j)}\to A$.
\end{enumerate}
If $m(j)=0$, $F_j^{\mathcal{A}}\in A$ is a constant. $\mathcal{A}$ is {\em finite} or {\em countable} if $|\mathcal{A}|$ is finite or countable, respectively.

Let $\mathcal{L}$ be a language and let $\mathcal{A}$ and $\mathcal{B}$ be $\mathcal{L}$-structures. By an {\em embedding} from $\mathcal{A}$ to $\mathcal{B}$, we mean an injection $f\colon |\mathcal{A}| \rightarrow |\mathcal{B}|$ such that for all $i \in I$ and $a_1,\dots,a_{n(i)} \in A$,
$$ R_i^{\mathcal{A}}(a_1,\dots,a_{n(i)}) \iff  R_i^{\mathcal{B}}(f(a_1),\dots,f(a_{n(i)})),$$
 and for all $j\in J$ and $a_1,\dots,a_{m(j)} \in A$,
$$ f(F_j^{\mathcal{A}}(a_1,\dots,a_{m(j)}))=F_j^{\mathcal{B}}(f(a_1),\dots,f(a_{m(j)})).$$ 
If $f$ is the inclusion map, then we say $\mathcal{A}$ is a {\em substructure} of $\mathcal{B}$, denoted by $\mathcal{A}\subseteq \mathcal{B}$. If $S\subseteq |\mathcal{B}|$, then the substructure of $\mathcal{B}$ {\em generated by} $S$ is the smallest substructure of $\mathcal{B}$ whose universe contains $S$ as a subset. In this case $S$ is called a {\em generating set} for the substructure generated by $S$. A substructure is {\em finitely generated} if it has a finite generating set.  

If there is an embedding from $\mathcal{A}$ to $\mathcal{B}$, we say $\mathcal{A}$ is {\em embeddable} to $\mathcal{B}$, denoted by $\mathcal{A} \leq \mathcal{B}$. Finally, if $f$ is bijective, we say it is an {\em isomorphism}; if there is an isomorphism between $\mathcal{A}$ and $\mathcal{B}$, we say they are {\em isomorphic}, and denote $\mathcal{A} \cong \mathcal{B}$. 

    For the rest of the discussion, fix a countable language $\mathcal{L}$. A class $\mathcal{K}$ of finite $\mathcal{L}$-structures is said to be a {\em Fra\"iss\'e class} if it satisfies the following conditions: 
    \begin{enumerate}
        \item Up to isomorphism, there are only countably many different structures in $\mathcal{K}$. 
        \item There are structures of arbitarily large finite cardinality in $\mathcal{K}$. 
        \item Hereditary property (HP): If $\mathcal{A} \in \mathcal{K}$, then every finite substructure of $\mathcal{A}$ is in $\mathcal{K}$. 
        \item Joint embedding property (JEP): If $\mathcal{A}, \mathcal{B} \in \mathcal{K}$, then there is a $\mathcal{C} \in \mathcal{K}$ such that $\mathcal{A} \leq \mathcal{C}$ and $\mathcal{B} \leq \mathcal{C}$. 
        \item Amalgamation property (AP): For any embeddings $f: \mathcal{A} \rightarrow \mathcal{B}$ and $g:\mathcal{A} \rightarrow \mathcal{C}$, there is a $\mathcal{D} \in \mathcal{K}$ and embeddings $r:\mathcal{B} \rightarrow \mathcal{D}$ and $s: \mathcal{C} \rightarrow \mathcal{D}$ such that $r\circ f=s\circ g$.   
    \end{enumerate}

A countable $\mathcal{L}$-structure $\mathcal{A}$ is {\em locally finite} if every finitely generated substructure of $\mathcal{A}$ is finite. The {\em age} of $\mathcal{A}$, denoted by ${\rm Age}(\mathcal{A})$, is the class of finite substructures of $\mathcal{A}$. We say $\mathcal{A}$ is {\em ultrahomogeneous} if every isomorphism between two finite substructures can be extended to an automorphism of $\mathcal{A}$. We call $\mathcal{A}$ a {\em Fra\"iss\'e structure} if it is locally finite, countably infinite and ultrahomogeneous. 

Fra\"iss\'e \cite{Fraisse} (c.f. \cite{Hodges}) gave a correspondence between Fra\"iss\'e classes and Fra\"iss\'e structures, as follows. If $\mathcal{K}$ is a Fra\"iss\'e class, then there is a unique Fra\"iss\'e structure $\mathcal{A}$, up to isomorphism, such that ${\rm Age}(\mathcal{A}) =\mathcal{K}$. This unique structure is called the {\em Fra\"iss\'e limit} of $\mathcal{K}$, denoted by $\mbox{Flim}\,{\mathcal{K}}$. Conversely, if $\mathcal{A}$ is a Fra\"iss\'e structure, then ${\rm Age}(\mathcal{A})$ is a Fra\"iss\'e class. 

We remark here that the original setting of Fra\"iss\'e theory allows a more general definition of Fra\"iss\'e classes than we presented above. For example, in the above definition of Fra\"iss\'e classes, every mention of finite substructures can be replaced by finitely generated substructures and condition (2) can be omitted. Correspondingly, in the definition of Fra\"iss\'e structures, the structure does not need to be locally finite or infinite. We define these notions as above just because the KPT correspondence requires them.

The countable language $\mathcal{L}$ we will consider in this paper are all {\em relational}, i.e., they consists only of relation symbols and have no function symbols. In this case, every structure is locally finite, and the age of a structure consists of exactly all finitely generated substructures.

    We note the following equivalent condition for a countable structure to be ultrahomogeneous. 

    \begin{definition}[\cite{Hodges}]
        Let $\mathcal{L}$ be a countable language. A countable $\mathcal{L}$-structure $\mathcal{A}$ is {\em weakly homogeneous} if for all $\mathcal{X}$ and $\mathcal{Y}$ in ${\rm Age}(\mathcal{A})$ with $\mathcal{X} \subseteq \mathcal{Y}$ and embedding $f$ from $\mathcal{X}$ to $\mathcal{A}$, $f$ can be extended to an embedding $g$ from $\mathcal{Y}$ to $\mathcal{A}$. 
    \end{definition}

    It turns out that for countable structures, ultrahomogeneity and weak homogeneity are equivalent. This equivalent condition will be used later. 

\subsection{The KPT correspondence} What has become known as the KPT correspondence was developed in Kechris--Pestov--Todorcevic \cite{KPT}. We review the basic definitions and results here.

A {\em Polish space} is a separable, completely metrizable topological space. A {\em Polish group} is a topological group whose topology is Polish. An example of a Polish group is the infinite permutation group $S_\infty$. The group $S_\infty$ consists of all permutations of $\mathbb{N}$, where the group operation is composition. The topology on $S_\infty$ is the pointwise convergence topology, where the topology on $\mathbb{N}$ is discrete. Equivalently, $S_\infty$ has the subspace topology inherited from $\mathbb{N}^\mathbb{N}$, where $\mathbb{N}^\mathbb{N}$ has the product topology. Both the topologies on $\mathbb{N}^\mathbb{N}$ and on $S_\infty$ are Polish.

Throughout the rest of this subsection we fix a countable language $\mathcal{L}$. For a countable $\mathcal{L}$-structure $\mathcal{A}$, we use ${\rm Aut}(\mathcal{A})$ to denote the group of all automorphisms of $\mathcal{A}$, where the group operation is composition. Considering the discrete topology on $|\mathcal{A}|$, and equipping $\mbox{\rm Aut}(\mathcal{A})$ with the topology of pointwise convergence, $\mbox{\rm Aut}(\mathcal{A})$ becomes a Polish group. In fact, it can be regarded as a closed subgroup of the permutation group $S_{|\mathcal{A}|}$. When $|\mathcal{A}|$ is infinite, $S_{|\mathcal{A}|}$ is isomorphic to $S_{\infty}$. Conversely, if $G$ is a closed subgroup of $S_{\infty}$, then there is a countable structure $\mathcal{A}$ such that $G$ and $\mbox{\rm Aut}(\mathcal{A})$ are isomorphic as topological groups.

A Fra\"iss\'e class $\mathcal{K}$ is said to be {\em ordered} if there is a binary relation symbol $<$ in $\mathcal{L}$ and, in every structure in $\mathcal{K}$, the interpretaton of $<$ is a linear order. The KPT correspondence unravels a deep connection between the property of a Fra\"iss\'e ordered class $\mathcal{K}$ and the extreme amenability of $\mbox{\rm Aut}(\mbox{Flim}\,\mathcal{K})$. This property is called the ({\em structural}) {\em Ramsey property}, which we define below.
    
Let $\mathcal{K}$ be a class of finite structures and let $\mathcal{A}, \mathcal{B} \in \mathcal{K}$. We use the following notation to denote the set of all substructures of $\mathcal{B}$ which are isomorphic to $\mathcal{A}$: 
    \begin{equation*}
        \begin{pmatrix}
            \mathcal{B} \\ 
            \mathcal{A}
        \end{pmatrix} = \{\mathcal{A'}\colon \mathcal{A'} \subseteq \mathcal{B}, \mathcal{A'} \cong \mathcal{A} \}. 
    \end{equation*}
    Then for $\mathcal{A} \leq \mathcal{B} \leq \mathcal{C}$ and $k\geq 2$, we use the {\em Erd\H{o}s-Rado arrow}, 
    \begin{equation*}
        \mathcal{C} \rightarrow (\mathcal{B})^{\mathcal{A}}_k
    \end{equation*}
    to mean that for any function $c: 
    \begin{pmatrix}
        \mathcal{C} \\ 
        \mathcal{A}
    \end{pmatrix} \rightarrow k$, there is a $\mathcal{B'} \in 
    \begin{pmatrix}
        \mathcal{C} \\
        \mathcal{B}
    \end{pmatrix}$ such that $c$ is constant on $
    \begin{pmatrix}
        \mathcal{B'} \\
        \mathcal{A}
    \end{pmatrix}.$
    Finally, we say $\mathcal{K}$ has the {\em Ramsey property} if for any $\mathcal{A} \leq \mathcal{B} \in \mathcal{K}$ and $k \geq 2$, there is a $\mathcal{C} \in \mathcal{K}$ such that $\mathcal{C} \rightarrow (\mathcal{B})^\mathcal{A}_k$. 

    \begin{theorem}[Kechris-Pestov-Todorcevic \cite{KPT}]\label{thm:KPT}
        Let $\mathcal{K}$ be a Fra\"iss\'e ordered class and let $\mathcal{A}$ be the Fra\"iss\'e limit of $\mathcal{K}$. Then the following are equivalent:
        \begin{enumerate}
            \item[\rm (1)] ${\rm Aut}(\mathcal{A})$ is extremely amenable. 
            \item[\rm (2)] $\mathcal{K}$ has the Ramsey property. 
        \end{enumerate}
    \end{theorem}

The simplest example of a Fra\"iss\'e class with the Ramsey property is the class of all finite sets. In fact, the Ramsey property for this class is equivalent to the classical Ramsey theorem. Also, the class of all finite linear orders has the Ramsey property, which is again equivalent to the classical Ramsey theorem; since the Fra\"iss\'e limit of this class is the countable dense linear order without endpoints, it follows from the above theorem that ${\rm Aut}(\mathbb{Q}, <)$ is extremely amenable. (Historically, Pestov \cite{Pestov98} proved the extreme amenability of $\mbox{\rm Aut}(\mathbb{Q},<)$ first, which motivated the above theorem.) Many more examples of Fra\"iss\'e ordered classes with the Ramsey property are now known and the automorphism groups of their Fra\"iss\'e limits are all extremely amenable by the KPT correspondence. For an excellent survey of these results, see Nguyen Van Th\'e \cite{NVT}. 

We will also use the following strengthening of Theorem~\ref{thm:KPT}, which is Theorem~3 of Nguyen Van Th\'e \cite{NVT}.

\begin{theorem}[\cite{NVT}]\label{thm:KPTstrong} Let $\mathcal{L}$ be a countable language, let $\mathcal{K}$ be a Fra\"iss\'e class of finite $\mathcal{L}$-structures, and let $\mathcal{A}$ be the Fra\"iss\'e limit of $\mathcal{K}$. Then the following are equivalent:
\begin{enumerate}
\item[\rm (1)] $\mbox{\rm Aut}(\mathcal{A})$ is extremely amenable.
\item[\rm (2)] $\mathcal{K}$ has the Ramsey property and every element of $\mathcal{K}$ is rigid, i.e., has trivial automorphism group.
\end{enumerate}
\end{theorem}

\subsection{Descriptive set theory and the KNT program\label{subsec:er}}
In this paper we will be discussing the complexity of equivalence relations in the framework of descriptive set theory. We review the basic concepts and notions in this subsection. For other undefined terms and concepts, see \cite{GaoBook}.

Let $X$ be a set and let $\mathcal{B}$ be a $\sigma$-algebra of subsets of $X$. $(X, \mathcal{B})$ is a {\em standard Borel space} if there is a Polish topology $\tau$ on $X$ such that $\mathcal{B}$ is the $\sigma$-algebra of Borel sets generated (via iterations of countable union, countable intersection, and complementation) by $\tau$. Any Polish space is a standard Borel space. If $X$ is a standard Borel space and $Y$ is a Borel subset of $X$, then $Y$ is a standard Borel space. A subset $A$ of a standard Borel space is {\em analytic} if there is a standard Borel space $Y$ and there is a Borel function $f\colon Y\to X$ with $A=f(Y)$. An equivalence relation $E$ on a standard Borel space $X$ is {\em Borel} (or {\em analytic}) if $E$ is a Borel (analytic) subset of $X^2$.

As mentioned in the Introduction, the class of all closed subgroups of $S_\infty$ coincides with the class of all non-Archimedean Polish groups \cite[Theorem 1.5.1]{BK}. This fact can be used to code all non-Archimedean Polish groups, or equivalently all automorphism groups of countable structures, into a standard Borel space, as follows. Let $\mathcal{F}(S_\infty)$ denote the space of all closed subsets of $S_\infty$. Equip $\mathcal{F}(S_\infty)$ with the $\sigma$-algebra of subsets generated by sets of the form 
$$\{F\in \mathcal{F}(S_\infty)\colon F \cap U \neq \emptyset\}, $$ where $U\subseteq S_\infty$ is a basic open set, i.e., of the form 
$$N_{\bar{a},\bar{b}}=\{g\in S_\infty\colon g(\bar{a})=\bar{b}\} $$
for some tuples $\bar{a}, \bar{b}$ of natural numbers of the same length. Then $\mathcal{F}(S_\infty)$ becomes a standard Borel space. 

Kechris--Nies--Tent \cite{KNT} initiated the following research program.

 \begin{programme}{\ }
        \begin{enumerate}
            \item For natural classes of closed subgroups of $S_\infty$, determine if they are Borel as subsets of $\mathcal{F}(S_\infty)$. 
            \item If a class is Borel, study the relative complexity of the topological isomorphism relation, using Borel reducibility $\leq_B$. 
        \end{enumerate}
     \end{programme}

It was shown in \cite{KNT} that the isomorphism relation on any Borel subsets of $\mathcal{F}(S_\infty)$ is an analytic equivalence relation. 

Next we show that our study of the non-Archimedean Polish groups that are extremely amenable fits well in the KNT program. Let $\mathsf{EA}$ denote the class of all closed subgroups of $S_\infty$ that are extremely amenable, and let $\cong_{\mathsf{EA}}$ denote the isomorphism relation on $\mathsf{EA}$. The following lemma shows that $\mathsf{EA}$ is a Borel subset of $\mathcal{F}(S_\infty)$.

     \begin{proposition}\label{thm:EA}
        The class $\mathsf{EA}$ of extremely amenable closed subgroups of $S_\infty$ is Borel. Consequently, $\cong_{\mathsf{EA}}$ is an analytic equivalence relation.
     \end{proposition}
     \begin{proof}
       Given any closed subgroup $G$ of $S_\infty$, by a well-known construction (see \cite[Theorem 4.1.4]{Hodges} and  of \cite[Section 2]{KPT}), we define a {\em canonical structure} $\mathcal{A}_G$ so that $\mbox{\rm Aut}(\mathcal{A}_G)=G$. To do this, consider the diagonal action of $G$ on $\mathbb{N}^n$ for all postive integers $n\geq 1$, which is defined by
$$ g\cdot (a_1,\dots, a_n)=(g(a_1),\dots, g(a_n)). $$
For each $n\geq 1$, let $\mathcal{O}_1^n, \mathcal{O}_2^n, \dots$ enumerate the set
$$ \left\{G\cdot (a_1,\dots, a_n)\colon (a_1,\dots, a_n)\in \mathbb{N}^n\right\}, $$
which is the set of all orbits of $G$ for the diagonal action on $\mathbb{N}^n$. Consider the language $\mathcal{L}$ consisting of relation symbols $R_{n,i}$ for $n, i\geq 1$. Each $R_{n,i}$ is an $n$-ary relation symbol. The structure $\mathcal{A}_G$ is an $\mathcal{L}$-structure whose universe is $\mathbb{N}$, with $R_{n,i}^{\mathcal{A}_G}=\mathcal{O}_i^n$. Then it is easy to check that $\mbox{\rm Aut}(\mathcal{A}_G)=G$. It is also straightforward to see that $\mathcal{A}_G$ is a Fra\"iss\'e structure, i.e., it is locally finite, countably infinite and ultrahomogeneous.

We note that the collection of all $\mathcal{L}$-structures with the universe $\mathbb{N}$ forms a Polish space $X_{\mathcal{L}}=\prod_{n,i\geq 1} X_{n,i}$, where each $X_{n,i}$ is a copy of the Cantor space $2^{\mathbb{N}^n}$. The map $G\mapsto \mathcal{A}_G$ from $\mathcal{F}(S_\infty)$ to $X_{\mathcal{L}}$ is Borel. In fact, a subbasic open set in $X_{\mathcal{L}}$ is determined by a tuple $\bar{a}=(a_1,\dots, a_n)$ being assigned value $0$ or value $1$. If $U$ is a basic open set in $X_{\mathcal{L}}$, then whether $\mathcal{A}_G\in U$ is determined by finitely many requirements of form $\exists g\in G\ g(\bar{a})=\bar{b}$ and its negation, where $\bar{a}, \bar{b}$ are tuples of the same length. Thus $\{G\in \mathcal{F}(S_\infty)\colon \mathcal{A}_G\in U\}$ is clearly a Borel set.

Let $\mathcal{K}$ be the age of $\mathcal{A}_G$. Then $\mathcal{A}_G=\mbox{Flim}\,\mathcal{K}$. By Theorem~\ref{thm:KPTstrong}, $G=\mbox{\rm Aut}(\mathcal{A}_G)=\mbox{\rm Aut}(\mbox{Flim}\,\mathcal{K})$ is extremely amenable if and only if $\mathcal{K}$ has the Ramsey property and every element of $\mathcal{K}$ is rigid. Thus to prove $\mathsf{EA}$ is Borel, it suffices to show both the following:
\begin{enumerate}
\item[(i)] the collection of $\mathcal{L}$-structures $\mathcal{A}\in X_{\mathcal{L}}$ where the age of $\mathcal{A}$ satisfies the Ramsey property is a Borel subset of $X_{\mathcal{L}}$;
\item[(ii)] the collection of $\mathcal{L}$-structures $\mathcal{A}\in X_{\mathcal{L}}$ where every finite substructure is rigid is a Borel subset of $X_{\mathcal{L}}$.
\end{enumerate}

Both (i) and (ii) refer to finite subsets of $\mathcal{A}$ (whose universe is $\mathbb{N}$). Thus, despite the tediousness of the statements of the Ramsey property and the rigidity, they can be expressed as Borel statements in the topology of $X_{\mathcal{L}}$. This completes the proof.
     \end{proof}
     
 After the completion of an earlier version of this paper, we learned from Aleksander Iwanow that Proposition~\ref{thm:EA} had been proved by Iwanow--Majcher-Iwanow in response to a question asked by Nies in the Logic Blog in 2020. The proof is included in the manuscript \cite{Nies} and is different from the one above.

\section{The Extreme Amenability and the Isomorphism Types of $\operatorname{Aut}(\mathbb{U}^<_\Delta)$\label{sec:n3}}


\subsection{$\Delta$-metric spaces}

 We work with the following notion. 

    \begin{definition}[\cite{EGLMM}]
        Let $\Delta$ be a nonempty set of positive real numbers. We call $\Delta$ a {\em distance value set} if for all $x, y\in \Delta$, 
        \begin{equation*}
           \min\{ x+y, \sup\Delta\} \in \Delta.
        \end{equation*}
    \end{definition}
Briefly speaking, a distance value set is exactly an initial segment of a subsemigroup of positive reals. Given a distance value set $\Delta$, a metric space $(X,d)$ is a {\em $\Delta$-metric space} if $d(x,y)\in\Delta$ for all $x\neq y\in X$. 

Fix a countable distance value set $\Delta$. We can view each countable $\Delta$-metric space as a countable structure, as follows. Let $\mathcal{L}_\Delta=\{R_s\colon s \in \Delta\}$, where each $R_s$ is a binary relation symbol. Given any countable $\Delta$-metric space $(X,d)$,  set $R^X_s(x,y)$ if and only if $d(x,y)=s$. Thus we turn $(X,d)$ into a countable $\mathcal{L}_\Delta$-structure $(X; (R^X_s)_{s\in \Delta})$. Let $\mbox{\rm Aut}(X,d)$, or simply $\mbox{\rm Aut}(X)$ if $d$ is understood, denote the automorphism group $\mbox{\rm Aut}(X; (R^X_s)_{s\in \Delta})$. Then $\mbox{\rm Aut}(X)$ is a Polish group.

Fix a countable distance value set $\Delta$. Let $\mathcal{K}_\Delta$ be the class of all finite $\Delta$-metric spaces. Then $\mathcal{K}_\Delta$ is a Fra\"iss\'e class. We denote the Fra\"iss\'e limit of $\mathcal{K}_\Delta$ by $\mathbb{U}_{\Delta}$ and call it the {\em $\Delta$-metric Urysohn space}. 

\subsection{The extreme amenability of $\mbox{\rm Aut}(\mathbb{U}_\Delta^<)$} 
        
    To apply the KPT correspondence, we need to consider a Fra\"iss\'e ordered class. Let $\Delta$ be a countable distance value set. Let $\mathcal{K}_\Delta^<$ be the class of all finite ordered $\Delta$-metric spaces, i.e., each element of $\mathcal{K}_\Delta^<$ is a structure $(X; (R_s^X)_{s\in \Delta}, <^X)$, where $<^X$ is a linear order on $X$. 
It is easy to see that $\mathcal{K}_\Delta^<$ is a Fra\"iss\'e (ordered) class. We denote its Fra\"iss\'e limit by $\mathbb{U}_\Delta^<$. It can be shown that $\mathbb{U}_\Delta^< \cong (\mathbb{U}_\Delta, <^{\mathbb{U}_\Delta})$ where $<^{\mathbb{U}_\Delta}$ is a linear order isomorphic to $(\mathbb{Q}, <)$ (see, e.g., Proposition 5.2 of \cite{KPT}). 

The following is a theorem of Ne\v{s}et\v{r}il \cite{Nes07}; Ngyen Van Th\'e \cite{NVTMAMS} stated the result more explicitly (as Theorem 13 of \cite{NVTMAMS}) and gave a full proof. 
    
\begin{theorem}[Ne\v{s}et\v{r}il] \label{thm:eaK<} Let $\Delta$ be a countable distance value set. Then $\mathcal{K}_\Delta^<$ has the Ramsey property.
\end{theorem}

By the KPT correspondence, we get that $\mbox{\rm Aut}(\mathbb{U}_\Delta^<)$ is extremely amenable for any countable distance value set $\Delta$.

\subsection{Isomorphism types of $\mbox{\rm Aut}(\mathbb{U}^<_\Delta)$} 
In this subsection we characterize the topological isomorphism relation of groups of the form $\mbox{\rm Aut}(\mathbb{U}^<_{\Delta})$ by a notion of equivalence between countable distance value sets. This notion is defined as follows.

\begin{definition}[\cite{EGLMM}]\label{def:CDVSequiv} Let $\Delta$ and $\Lambda$ be distance value sets. 
\begin{enumerate}
\item[(i)] A triple $(x,y,z)$ of positive real numbers is a {\em $\Delta$-triangle} if $x,y, z\in \Delta$ and $|x-y|\leq z\leq x+y$.
\item[(ii)] We say that $\Delta$ and $\Lambda$ are {\em equivalent} if there is a bijection $f\colon \Delta\to \Lambda$ such that for any positive real numbers $x,y, z$, $(x,y,z)$ is a $\Delta$-triangle if and only if $(f(x), f(y), f(z))$ is a $\Lambda$-triangle.
\end{enumerate}
\end{definition}


Our main theorem of this subsection is the following.

    \begin{theorem}\label{thm:class}
        Let $\Delta$ and $\Lambda$ be countable distance value sets. Then the following are equivalent: 
        \begin{enumerate}
            \item[\rm (i)] ${\rm Aut}(\mathbb{U}_\Delta^<)$ and ${\rm Aut}(\mathbb{U}_\Lambda^<)$ are isomorphic as topological groups. 
            \item[\rm (ii)] $\Delta$ and $\Lambda$ are equivalent. 
        \end{enumerate}
    \end{theorem}

The rest of this subsection is devoted to a proof of Theorem~\ref{thm:class}. We recall the following theorem from \cite{EGLMM}.

 \begin{theorem}[\cite{EGLMM}]\label{thm:char}
        Let $\Delta$ and $\Lambda$ be countable distance value sets. Then the following are equivalent: 
        \begin{enumerate}
            \item[\rm (i)] ${\rm Aut}(\mathbb{U}_\Delta)$ and ${\rm Aut}(\mathbb{U}_\Lambda)$ are isomorphic as topological groups. 
            \item[\rm (ii)] $\Delta$ and $\Lambda$ are equivalent. 
        \end{enumerate}
    \end{theorem}
    
    In fact, the outline of the proof below follows the proof of Theorem~\ref{thm:char} in \cite{EGLMM}. 

To see that (ii) implies (i) in Theorem~\ref{thm:class}, fix a bijection $f$ from $\Delta$ to $\Lambda$ which witnesses the equivalence of the two distance value sets. Let $d$ be the metric on $\mathbb{U}^<_\Delta$. Define a new metric $d_f$ on $\mathbb{U}^<_\Delta$ by setting $d_f(x,y)=f(d(x,y))$. Then $(\mathbb{U}^<_\Delta, d_f)$ is isometrically isomorphic to $\mathbb{U}^<_\Lambda$, since $(\mathbb{U}^<_\Delta, d_f)$ is a Fra\"iss\'e structure and its age is exactly the class of all finite ordered $\Lambda$-metric spaces. Now the identity map from $(\mathbb{U}^<_\Delta,d)$ to $(\mathbb{U}^<_\Delta, d_f)$ induces a topological isomorphism from $\operatorname{Aut}(\mathbb{U}^<_\Delta)$ to $\operatorname{Aut}(\mathbb{U}^<_\Lambda)$. 

To prove the other direction, we need some results due to Slutsky \cite{Slutsky}. We work with the following concept from \cite{Slutsky} in the Fra\"iss\'e theory of finite metric spaces. 

\begin{definition} Let $\mathcal{A}, \mathcal{B}, \mathcal{C}$ be finite metric spaces.  For simplicity assume $\mathcal{A}$ is a substructure of both $\mathcal{B}$ and $\mathcal{C}$. Also assume $|\mathcal{B}|\cap |\mathcal{C}|=|\mathcal{A}|$. Define a finite metric space $\mathcal{D}$ as follows. Let $|\mathcal{D}|$ = $|\mathcal{B}| \cup |\mathcal{C}|$. For $x,y\in |\mathcal{D}|$, define $d_{\mathcal{D}}(x,y)$ as follows. If $x,y \in |\mathcal{B}|$, then let $d_{\mathcal{D}}(x,y)=d_{\mathcal{B}}(x,y)$. Similarly, if $x,y\in \mathcal{C}$, then let $d_{\mathcal{D}}(x,y)=d_{\mathcal{C}}(x,y)$. If $x\in |\mathcal{B}|$ and $y\in |\mathcal{C}|$, then let
    \begin{equation*}
        d_{\mathcal{D}}(x,y)= \left \{
            \begin{array}{ll}
                \min\{d_\mathcal{B}(x,z)+d_\mathcal{C}(z,y)\colon  z\in |\mathcal{A}|\}, & \mbox{ if $|\mathcal{A}| \neq \varnothing$,} \\
                \mbox{diam}(\mathcal{B})+\mbox{diam}(\mathcal{C}),  & \mbox{ if $|\mathcal{A}| = \varnothing$.}
            \end{array}\right .
    \end{equation*}
Then $\mathcal{D}$ is called the {\em free amalgam} of $\mathcal{B}$ and $\mathcal{C}$ over $\mathcal{A}$.
\end{definition}

Since (ordered) metric spaces can be regarded as relational structures, for an (ordered) metric space $\mathcal{B}$, a substructure $\mathcal{A}$ of $\mathcal{B}$ generated by a set $A\subseteq |\mathcal{B}|$ satisfies $|\mathcal{A}|=A$. Thus we will identify $\mathcal{A}$ with $A$, and speak of $A$ as a substructure. This does not cause confusion, and will simplify our notation.

    \begin{definition}
        Let $A$ be a metric space and $p$ a partial isometry on $A$. A point $x \in {\rm dom}(p)$ is called {\em periodic} if there is a natural number $n>0$ such that \[x,p(x),\dots,p^n(x)\in {\rm dom}(p) \ {\rm and} \ p^n(x)=x.\] The set of all periodic points of $p$ is denoted by $Z(p)$. A point $x \in {\rm dom}(p)$ is called a {\em fixed point} if $p(x)=x$. The set of all fixed points of $p$ is denoted by $F(p)$. 
    \end{definition}

    \begin{lemma}[\cite{Slutsky}]\label{lemma:Extensions}
        Let $D$ be a finite ordered metric space. Let $p_1$ and $p_2$ be order-preserving partial isometries of $D$ such that $Z(p_1)=F(p_1)$ and $Z(p_2)=F(p_2)$. Suppose $F(p_1) \cap F(p_2) \neq \emptyset$. Then there exist a finite ordered metric space $D'$ extending $D$, order-preserving  partial isometries $q_1$ and $q_2$ of $D'$ extending $p_1,p_2$ respectively, and an element $w \in \mathsf{F}(s,t)$, the free group generated by ${s,t}$, such that 
        \begin{enumerate}
           \item $Z(q_1)=Z(p_1)$ and $Z(q_2)=Z(p_2)$;
           \item $D \subseteq {\rm dom}(q_1)$;
           \item ${\rm dom}(q_1) \cup w(q_1,q_2)(D)$ is the free amalgam of ${\rm dom}(q_1)$ and \linebreak $w(q_1,q_2)(D)$ over $F(p_1) \cap F(p_2)$. 
        \end{enumerate}
        Moreover, the distances in $D'$ are taken from the additive semigroup generated by the distances in $D$. 
     \end{lemma}

     This result was originally proved for metric spaces as Theorem~4.12 of \cite{Slutsky}. Nevertheless, as Slutsky remarked (Remark 4.13 of \cite{Slutsky}), the above ordered version also holds.

We use the following notation in our further discussions. For a countable structure $\mathcal{A}$ and subset $B \subseteq |\mathcal{A}|$, we use ${\rm Aut}_B(\mathcal{A})$ to denote the subgroup of $\mbox{\rm Aut}(\mathcal{A})$ consisting of the automorphisms of $\mathcal{A}$ which fix each element of $B$. Note that if $B$ is finite, then ${\rm Aut}_B(\mathcal{A})$ is an open subgroup of $\mbox{\rm Aut}(\mathcal{A})$. 

 The following two lemmas are key technical results in our proof. The combination of their unordered versions was a theorem of Slutsky \cite{Slutsky}. 

     \begin{lemma}\label{lemma:density}
        Let $\Delta$ be a countable distance value set, and let $A,B \subseteq \mathbb{U}_{\Delta}^<$ be finite subsets. Suppose $A \cap B \neq \varnothing$. Then \[\langle{\rm Aut}_A(\mathbb{U}_\Delta^<), {\rm Aut}_B(\mathbb{U}_\Delta^<) \rangle\ is \ dense \ in \ {\rm Aut}_{A\cap B}(\mathbb{U}_\Delta^<). \]
     \end{lemma}
     \begin{proof}
        For notational simplicity, we assume $A \cap B =\{x_0\}$. Suppose $A=\{a_0,\dots,a_n\}$ and $B=\{b_0,\dots,b_m\}$. Again, for definiteness and notational simplicity, assume $m=n$ and $a_0=b_0=x_0 < a_1 <b_1 <\dots<a_n <b_n$. Note that in general, $A \cap B =\{x_0,x_1,\dots, x_k\}$ where $x_0<x_1<\cdots<x_k$ and the rest of the points in $A,B$ can be in the gaps between $x_i$s including below $x_0$ and above $x_k$. For the general case, the argument is similar.

        Let $C$ be a finite subset of $\mathbb{U}_\Delta^<$ and let $p \in {\rm Aut}_{A\cap B}(\mathbb{U}_\Delta^<)$. It suffices to show that there is a $q \in \langle {\rm Aut}_A(\mathbb{U}_\Delta^<), {\rm Aut}_B(\mathbb{U}_\Delta^<) \rangle$ such that $p \upharpoonright C = q\upharpoonright C$. Without loss of generality, we may assume $A \subseteq C$ and $B \subseteq C$. Let $D = C \cup p(C)$. We find (to be explicitly defined later)
$$q=w(q_1,q_2)^{-1}g_1^{-1}\dots g_{2n}^{-1}rg_{2n}\dots g_1w(q_1,q_2), $$
where $g_{2i-1}\in {\rm Aut}_{A}(\mathbb{U}_\Delta^<)$, $r\in {\rm Aut}_{A}(\mathbb{U}_\Delta^<)$ and $g_{2i}\in {\rm Aut}_{B}(\mathbb{U}_\Delta^<)$. 


As a substructure of $\mathbb{U}^<_\Delta$, $D$ is a finite ordered metric space. Define two order-preserving partial isometries $p_1,p_2$ of $D$ by letting $p_1(x)=x$ for all $x \in A$ and $p_2(x)=x$ for all $x \in B$. Now applying Lemma~\ref{lemma:Extensions} to $D$ and $p_1,p_2$, we get a finite ordered metric space $D'$ extending $D$, order-preserving partial isometries $q_1,q_2$ of $D'$ which are extensions of $p_1,p_2$, respectively, and a word $w \in \mathsf{F}(s,t)$ such that 
\begin{enumerate}
\item $Z(q_1)=Z(p_1)=F(p_1)=A$ and $Z(q_2)=Z(p_2)=F(p_2)=B$;
\item $D \subseteq {\rm dom}(q_1)$;
\item ${\rm dom}(q_1) \cup w(q_1,q_2)(D)$ is the free amalgam of $\mbox{dom}(q_1)$ and $w(q_1,q_2)(D)$ over $\{x_0\}$. 
\end{enumerate}
For notational simplicity we write $w(q_1,q_2)$ as $w$. 
      
We note that $(A \cup B) \cap w(D) = \{x_0\}$. This is because $x_0\in A\cup  B\subseteq C\subseteq D\subseteq \mbox{dom}(q_1)$ and by (3), $\mbox{dom}(q_1)\cup w(D)$ is the free amalgam of $\mbox{dom}(q_1)$ and $w(D)$ over $\{x_0\}$. 

Let $w_1(D)$ be an isomorphic copy of $w(D)$ so that $w_1(D)$ and $w(D)$ have exactly one common point $x_0$. For each $y\in D$, the point in $w_1(D)$ corresponding to $w(y)\in w(D)$ will be denoted by $w_1(y)$. Let $D_1'$ be the free amalgam of $D'$ and $w_1(D)$ over $\{x_0\}$. Figure~\ref{fig:1} illustrates the containment relationship among the sets considered so far. 

\begin{figure}[h]
\begin{tikzpicture}[scale=0.1]

\draw (0,0) to (50,50) to (0,100) to (0,0);
\draw (10,10) to (10,40) to (50,50) to (90,40) to (90,10) to (50,50);
\draw (6, 94) to (6,61) to (50,50);
\draw (10, 60) to (10,90);
\draw (18, 58) to (18,82);
\draw (26, 56) to (26,62) to (50,50);
\draw (26, 74) to (26, 68) to (50,50);
\draw[->] (10,93) to (8,90);
\filldraw[fill=black, very thick] (50,50) circle (0.3);

\node[left] at (8, 50) {$D'$};
\node[left] at (25,30) {$w(D)$};
\node[left] at (88, 31) {$w_1(D)$};
\node[left] at (16, 73) {$D$};
\node[left] at (24, 67) {$C$};
\node[left] at (31, 69) {$A$};
\node[left] at (31, 58) {$B$};
\node[left] at (26,95) {$\mbox{dom}(q_1)$};
\node[left] at (56,52) {$x_0$};

\end{tikzpicture}
\caption{The composition of $D_1'$. \label{fig:1}}
\end{figure}
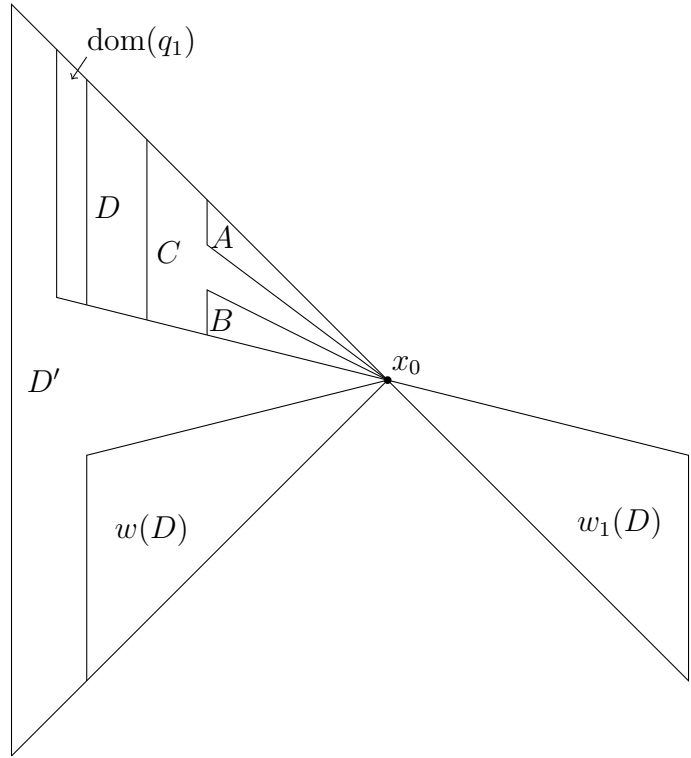
      
Extend the ordering on $D'$ to $D_1'$ so that the following hold:
    \begin{enumerate}
       \item[(a)] for any $y_1,y_2 \in D$, $w_1(y_1) < w_1(y_2)$ if and only if $w(y_1) < w(y_2)$;
       \item[(b)] for any $u \in A \cup B$ and $y \in D$, if $w(y) < b_n=\max(A\cup B)$, then $u<w_1(y)$ if and only if $u<w(y)$;
       \item[(c)] for any $y \in D$, if $w(y) > b_n$, then $a_n < w_1(y) <b_n$. 
    \end{enumerate}
    Such an extension clearly exists. 

Now define $g_1\colon A\cup w(D)\to A\cup w_1(D)$ by letting $g_1(a) = a$ for all $a\in A$ and $g_1(w(y)) = w_1(y)$ for all $y\in D$. Then $g_1$ is order-preserving by the above requirements (a) through (c). To see that $g_1$ is a partial isometry, let $a \in A$ and $y\in D$. Since $\mbox{dom}(q_1)\cup w(D)$ is the free amalgam of $\mbox{dom}(q_1)$ and $w(D)$ over $\{x_0\}$ and $D_1'$ is the free amalgam of $D'$ and $w_1(D)$ over $\{x_0\}$, we have \[d(a, w(y))=d(a,x_0)+d(x_0,w(y))=d(a,x_0)+d(x_0,w_1(y))=d(a,w_1(y)).\] 
    Intuitively, $g_1$ sends $w(y)$ to a copy which is smaller than $b_n$ and fixes $A$ pointwise. Consequently, we have $w_1(D)<b_n$ and, for any $y \in D$, $w_1(y) <x_0$ if and only if $w(y)<x_0$.
    
Let $w_2(D)$ be a fresh isomorphic copy of $w(D)$ such that $w_2(D)$ and $w(D)$ has exactly one common point $x_0$. Let $D_2'$ be the free amalgam of $D_1'$ and $w_2(D)$ over $\{x_0\}$. Extend the ordering on $D_1'$ to $D_2'$ so that the following hold:
 \begin{enumerate}
       \item[(d)] for any $y_1,y_2 \in D$, $w_2(y_1) < w_2(y_2)$ if and only if $w_1(y_1) < w_1(y_2)$;
       \item[(e)] for any $u \in A \cup B$ and $y \in D$, if $w(y) < a_n$, then $u<w_2(y)$ if and only if $u<w_1(y)$;
       \item[(f)] for any $y \in D$, if $w(y) > a_n$, then $b_{n-1} < w_2(y) <a_n$. 
    \end{enumerate}
Define an order-preserving partial isometry $g_2\colon B\cup w_1(D)\to B\cup w_2(D)$ by letting $g_2(b)=b$ for all $b\in B$ and $g_2(w_1(y))=w_2(y)$ for all $y\in D$. Moreover, $w_2(D) < a_n$ and, for any $y \in D$, $w_2(y) <x_0$ if and only if $w_1(y)<x_0$.
    
Repeating this construction, we obtain, for $1\leq i \leq 2n$, $D_i'=D' \cup w_1(D) \cup \dots \cup w_i(D)$, which can be viewed as the free amalgam of $w_i(D)$ and the remaining part over $\{x_0\}$, an ordering on $D_{2n}'$ extending that on $D'$, and order-preserving partial isometries $g_1, g_2, g_3, \dots,g_{2n}$ such that 
\begin{enumerate}
\item[(i)] for any $1\leq i\leq n$, $g_{2i-1}(a)=a$ for all $a\in A$, and $g_{2i}(b)=b$ for all $b\in B$;
\item[(ii)] for any $1\leq i\leq n$ and $y\in D$, $g_{2i}(w_{2i-1}(y))=w_{2i}(y)$, and $g_{2i-1}(w_{2i-2}(y))=w_{2i-1}(y)$ if $i>1$.
\end{enumerate}
Moreover, we have $w_{2n}(D) <a_1$ and, for any $y \in D$, $w_{2n}(y) <x_0$ if and only if $w(y)<x_0$. 

Define a map $$r\colon A\cup w_{2n}(C)\to A\cup w_{2n}(p(C))$$ by letting $r(a)=a$ for all $a\in A$ and $r(w_{2n}(c))=w_{2n}(p(c))$ for all $c\in C$. Then for $a \in A$ and $c \in C$,  we have
    \begin{align*}
    d(a,w_{2n}(c)) &=d(a,x_0)+d(x_0,w_{2n}(c))=d(a,x_0)+d(x_0,w(c))\\&=d(a,x_0)+d(x_0,c)=d(a,x_0)+d(x_0,p(c))\\&=d(a,x_0)+d(x_0,w(p(c))) \\
& =d(a,x_0)+d(x_0,w_{2n}(p(c)))=d(a,w_{2n}(p(c))) \\
&=d(r(a),r(w_{2n}(c))).
    \end{align*} 
Thus $r$ is a partial isometry. Note that $w$ and $p$ are order-preserving and each fixes $x_0$. Thus for any $c\in C$, we have 
    \begin{align*}
      w_{2n}(c)<x_0 &\iff w(c)<x_0 \iff c<x_0 \\&\iff p(c) <x_0 \iff w(p(c))<x_0 \\&\iff w_{2n}(p(c))<x_0. 
    \end{align*}
    Since for all $y \in D$, $w_{2n}(y)<a_1$, we conclude that $r$ is an order-preserving partial isometry.

By Lemma~\ref{lemma:Extensions} and our construction, the distances in $D_{2n}'$ are all taken from the additive semigroup generated by the distances in $D$. Now redefine the metric on $D_{2n}'$ by taking the minimum of the original distance value and $\sup\Delta$. Call the resulting metric space $D''$. Since $D$ is a $\Delta$-metric space, it follows that $D''$ is a $\Delta$-metric space, and, as an ordered metric space, $D''$ is an extension of $D$. Moreover, each of the order-preserving partial isometries $q_1, q_2, g_1, g_2, \dots, g_{2n}, r$ continues to be order-preserving partial isometries of $D''$. 

Now by the weak homogeneity of $\mathbb{U}_{\Delta}^<$, we can extend $D$ to an isomorphic copy of $D''$ as a substructure of $\mathbb{U}_\Delta^<$. By the ultrahomogeneity of $\mathbb{U}_\Delta^<$, we can extend all the above order-preserving partial isometries to full automorphisms of $\mathbb{U}_{\Delta}^<$. Denote the resulting automorphisms by the same notation. Then for $1\leq i\leq 2n$, we have $r, q_1,g_i \in {\rm Aut}_A(\mathbb{U}_{\Delta}^<)$ if $i$ is odd, and $q_2,g_i \in {\rm Aut}_B(\mathbb{U}_{\Delta}^<)$ if $i$ is even. Finally, for any $c \in C$, \[p(c)=w(q_1,q_2)^{-1}g_1^{-1}\dots g_{2n}^{-1}rg_{2n}\dots g_1w(q_1,q_2)(c).\]
Thus if we let 
$$q=w(q_1,q_2)^{-1}g_1^{-1}\dots g_{2n}^{-1}rg_{2n}\dots g_1w(q_1,q_2), $$
then $q\in \langle {\rm Aut}_A(\mathbb{U}_\Delta^<), {\rm Aut}_B(\mathbb{U}_\Delta^<) \rangle$ and $p\upharpoonright C=q\upharpoonright C$ as required.
     \end{proof}

     \begin{lemma}\label{theorem:density}
        Let $\Delta$ be a countable distance value set, and let $A,B \subseteq \mathbb{U}_{\Delta}^<$ be finite subsets. Suppose $A \cap B = \varnothing$. Then \[\langle {\rm Aut}_A(\mathbb{U}_{\Delta}^<), {\rm Aut}_B(\mathbb{U}_{\Delta}^<) \rangle  \ is \ dense \ in \  {\rm Aut}(\mathbb{U}_{\Delta}^<).\]
\end{lemma}    

\begin{proof}   Let $p \in {\rm Aut}(\mathbb{U}_{\Delta}^<)$ and let $C \subseteq \mathbb{U}_{\Delta}^<$ be a finite subset. We show that there is a $q \in \langle {\rm Aut}_A(\mathbb{U}_{\Delta}^<), {\rm Aut}_B(\mathbb{U}_{\Delta}^<) \rangle$ such that $p\upharpoonright C=q\upharpoonright C$. Again we may assume $A\cup B \subseteq C$. Consider the ordered $\Delta$-metric space $C \cup p(C)$. Extend it to an ordered $\Delta$-metric space $C \cup p(C) \cup \{x_0\}$, where the metric is defined by 
$$d(x_0,y)=\mbox{diam}(C\cup p(C))=\max\{d(y_1,y_2)\colon y_1,y_2 \in C \cup p(C)\}$$
for any $y\in C\cup p(C)$, and the ordering is defined by $x_0 <y$ for all $y\in C\cup p(C)$.
           Such an extension can be found in $\mathbb{U}_{\Delta}^<$ by the weak homogeneity of $\mathbb{U}^<_\Delta$, and so we may assume $x_0 \in \mathbb{U}_{\Delta}^<$. Now consider $A' = A \cup \{x_0\}$, $B' = B \cup \{x_0\}$ and let $p' \in {\rm Aut}_{x_0}(\mathbb{U}_\Delta^<)$ be an extention of $p\upharpoonright C \cup \{(x_0,x_0)\}$. Applying Lemma~\ref{lemma:density} to $A', B'$ and $p'$, we get a $q \in \langle {\rm Aut}_{A'}(\mathbb{U}_{\Delta}^<), {\rm Aut}_{B'}(\mathbb{U}_{\Delta}^<) \rangle$ such that $q\upharpoonright C\cup\{x_0\}=p'\upharpoonright C\cup\{x_0\}$. Then in particular $q \in \langle {\rm Aut}_A(\mathbb{U}_{\Delta}^<), {\rm Aut}_B(\mathbb{U}_{\Delta}^<) \rangle$ and $q\upharpoonright C=p\upharpoonright C$ as required. 
        \end{proof}
 
We also need the following two lemmas about proper open subgroups of $\mbox{\rm Aut}(\mathbb{U}^<_\Delta)$.    

     \begin{lemma}\label{lemma:f_orbit}
        Let $V$ be a proper open subgroup of ${\rm Aut}(\mathbb{U}_{\Delta}^<)$. Then there exists an element $x \in \mathbb{U}_{\Delta}^<$ such that $V\cdot x$ is finite.
        \begin{proof}
           Toward a contradiction, assume $V\cdot x$ is infinite for all $x\in \mathbb{U}_{\Delta}^<$. We prove that $V$ is dense. Since $V$ is also closed, it follows that $V = {\rm Aut}(\mathbb{U}_{\Delta}^<)$, contradicting the properness of $V$. 

Note that a basic open neighborhood of the identity element of $\mbox{\rm Aut}(\mathbb{U}^<_\Delta)$ has the form ${\rm Aut}_{A}(\mathbb{U}_{\Delta}^<)$, where $A$ is a finite subset of $\mathbb{U}^<_\Delta$. Since $V$ is an open subgroup, there is a finite $A\subseteq \mathbb{U}^<_\Delta$ such that ${\rm Aut}_{A}(\mathbb{U}_{\Delta}^<) \subseteq V$. By Neumann's lemma (see Corollary 4.2.2 of Hodges \cite{Hodges}), there exists $g \in V$ such that $g(A)\cap A = \varnothing$. 
           
By Lemma~\ref{theorem:density}, $\langle {\rm Aut}_{A}(\mathbb{U}_{\Delta}^<), {\rm Aut}_{g(A)}(\mathbb{U}_{\Delta}^<) \rangle$ is dense in ${\rm Aut}(\mathbb{U}_{\Delta}^<)$. This group is contained in $V$, and thus $V$ is dense. 
        \end{proof}
     \end{lemma}

     \begin{lemma}\label{lem:maxopen}
        Let $V$ be a proper open subgroup of ${\rm Aut}(\mathbb{U}_{\Delta}^<)$. The following are equivalent:
        \begin{enumerate}
           \item[\rm (1)] There exists a unique $x \in \mathbb{U}_{\Delta}^<$ such that $V = {\rm Aut}_{\{x\}}(\mathbb{U}_{\Delta}^<)$. 
           \item[\rm (2)] $V$ is maximal among proper closed subgroups of ${\rm Aut}(\mathbb{U}_{\Delta}^<)$. 
        \end{enumerate}
        \end{lemma}
\begin{proof}
           First suppose $V = {\rm Aut}_{\{x\}}(\mathbb{U}_{\Delta}^<)$. Assume $H$ is a closed subgroup of $\mbox{\rm Aut}(\mathbb{U}^<_\Delta)$ such that $V\subsetneq H\subsetneq \mbox{\rm Aut}(\mathbb{U}^<_\Delta)$.  Let $g \in H \setminus V$. Then $g(x) \neq x$ and ${\rm Aut}_{\{g(x)\}}(\mathbb{U}_{\Delta}^<) = g{\rm Aut}_{\{x\}}(\mathbb{U}_{\Delta}^<)g^{-1}=gVg^{-1} \leq H$. By Lemma~\ref{theorem:density}, $\langle {\rm Aut}_{\{g(x)\}}(\mathbb{U}_{\Delta}^<), {\rm Aut}_{\{x\}}(\mathbb{U}_{\Delta}^<) \rangle$ is dense in ${\rm Aut}(\mathbb{U}_{\Delta}^<)$. Since it is a subgroup of $H$ and $H$ is closed, we have $H= {\rm Aut}(\mathbb{U}_{\Delta}^<)$, contradicting the properness of $H$. 
     
           For the converse, assume (2) holds. By Lemma~\ref{lemma:f_orbit}, there exists an element $x \in \mathbb{U}_{\Delta}^<$ such that $V\cdot x$ is finite. We show that $V\subseteq \mbox{\rm Aut}_{\{x\}}(\mathbb{U}^<_\Delta)$. Otherwise, let $g \in V$ be such that $g(x) \neq x$. Then either $g(x)>x$ or $g(x)<x$. Since $g$ is order-preserving, we get that either $g^{n+1}(x) > g^n(x)$ for all natural numbers $n$, or $g^{n+1}(x)<g^n(x)$ for all natural numbers $n$. In either case, $V\cdot x$ is infinite. 
           
It is clear that $\mbox{\rm Aut}_{\{x\}}(\mathbb{U}^<_\Delta)\neq \mbox{\rm Aut}(\mathbb{U}^<_\Delta)$. Thus by the the maximality of $V$, we have
$V=\mbox{\rm Aut}_{\{x\}}(\mathbb{U}^<_\Delta)$.
     
           For the uniqueness, assume ${\rm Aut}_{\{x\}}(\mathbb{U}_{\Delta}^<) = {\rm Aut}_{\{y\}}(\mathbb{U}_{\Delta}^<)$ but $x \neq y$. Then by Lemma~\ref{theorem:density}, $\langle {\rm Aut}_{\{x\}}(\mathbb{U}_{\Delta}^<), {\rm Aut}_{\{y\}}(\mathbb{U}_{\Delta}^<) \rangle$ is dense. But $$\langle {\rm Aut}_{\{x\}}(\mathbb{U}_{\Delta}^<), {\rm Aut}_{\{y\}}(\mathbb{U}_{\Delta}^<) \rangle={\rm Aut}_{\{x\}}(\mathbb{U}_{\Delta}^<) = {\rm Aut}_{\{y\}}(\mathbb{U}_{\Delta}^<),$$ and thus they are both the whole group, which is impossible. 
        \end{proof}
     
     We are now ready to finish the proof of Theorem~\ref{thm:class}. Recall that the remaining direction is (i) implying (ii).

        Let $f\colon \mbox{\rm Aut}(\mathbb{U}^<_\Delta)\to\mbox{\rm Aut}(\mathbb{U}^<_\Lambda)$ be a topological group isomorphism. By Lemma~\ref{lem:maxopen}, for each $x \in \mathbb{U}_{\Delta}^<$, there is a unique $y \in \mathbb{U}_{\Lambda}^<$ such that $f({\rm Aut}_{\{x\}}(\mathbb{U}_{\Delta}^<))={\rm Aut}_{\{y\}}(\mathbb{U}_{\Lambda}^<)$. We write $y=g(x)$. The map $g$ is well-defined and bijective. 

        Note that for any $x,y,x',y'\in\mathbb{U}_{\Delta}^<$, ${\rm Aut}_{\{x,y\}}(\mathbb{U}_{\Delta}^<)$ and ${\rm Aut}_{\{x',y'\}}(\mathbb{U}_{\Delta}^<)$ are conjugate in $\mbox{\rm Aut}(\mathbb{U}^<_\Delta)$ if and only if $d(x,y)=d(x',y')$. Also note that
$$ {\rm Aut}_{\{x,y\}}(\mathbb{U}_{\Delta}^<)=\mbox{\rm Aut}_{\{x\}}(\mathbb{U}^<_\Delta)\cap \mbox{\rm Aut}_{\{y\}}(\mathbb{U}^<_\Delta) $$
and similarly for $\mbox{\rm Aut}_{\{x',y'\}}(\mathbb{U}^<_\Delta)$. Thus by the definition of $g$, we have that for any $x,y,x',y' \in \mathbb{U}_{\Delta}^<$,
$$ d(x,y)=d(x',y') \iff d(g(x),g(y)) = d(g(x'),g(y')).$$
        
Now for any $d \in \Delta$, we define $\theta(d) \in \Lambda$ by finding $x,y \in \mathbb{U}_{\Delta}^<$ such that $d(x,y)=d$, and then setting $\theta(d)=d(g(x),g(y))$. By the above observation, $\theta$ is well-defined and bijective. We show that $\theta$ witnesses the equivalence of $\Delta $ and $ \Lambda$. Suppose $(d_1,d_2,d_3)$ is a $\Delta$-triangle. There exist $x,y,z \in \mathbb{U}_{\Delta}^<$ such that the distances among them are exactly $(d_1,d_2,d_3)$. Then the distaces among $g(x),g(y),g(z)$ are $(\theta(d_1),\theta(d_2),\theta(d_3))$. Thus $(\theta(d_1),\theta(d_2),\theta(d_3))$ is a $\Lambda$-triangle. The converse also holds.

The proof of Theorem~\ref{thm:class} is complete.

\section{Proof of the Main Theorem\label{sec:n4}}
   
In this final section we prove Theorem~\ref{thm:main7} by determining the complexity of the equivalence (see Definition~\ref{def:CDVSequiv}) between countable distance value sets in the Borel reducibility hierarchy. We first establish some notation.

\begin{definition}
        Let $\mathsf{CDVS}\subseteq \mathbb{R}^{\mathbb{N}}$ consist of all sequences $(d_i)_{i\in\mathbb{N}}$ of real numbers such that
\begin{enumerate}
\item[(a)] $d_i\geq 0$ for all $i$; the set $\{i\in\mathbb{N}\colon d_i=0\}$ is infinite; $d_i>0$ for some $i$;
\item[(b)] if $d_i, d_j>0$, then $d_i=d_j$ if and only if $i=j$;
\item[(c)] if $\sup_kd_k<\infty$, then $\sup_k d_k=d_{i}$ for some $i$;
\item[(d)] for any $i,j\geq 0$, if $d_i+d_j<\sup_kd_k$ then there is $\ell$ such that $d_{\ell}=d_i+d_j$.
\end{enumerate}
\end{definition}

Intuitively, $\mathsf{CDVS}$ consists of codes of all countable distance value sets, which are enumerated as sequences of real numbers without repetitions. We include zero entries in the codes to accommodate finite distance value sets. $\mathsf{CDVS}$ is a Borel subset of $\mathbb{R}^{\mathbb{N}}$, and hence it is a standard Borel space. 

In our discussions below, we will not distinguish a countable distance value set $\Delta$ from a code of $\Delta$ as an element of $\mathsf{CDVS}$.

\begin{definition}\label{def:er} 
Let $\approx$ be the equivalence relation on $\mathbb{R}^{\mathbb{N}}$ defined by
$$\begin{array}{rcl} (c_i)_i\approx (d_i)_i&\iff& \exists g\in S_\infty\ [\forall i\ (c_i=0\longleftrightarrow d_{g(i)}=0) \mbox{ and } \\
& & \forall i, j, k\ (|c_j-c_k|\leq c_i\leq c_j+c_k\longleftrightarrow \\
& & |d_{g(j)}-d_{g(k)}|\leq d_{g(i)}\leq d_{g(j)}+d_{g(k)})].
\end{array}
$$
For any $A\subseteq \mathbb{R}^{\mathbb{N}}$, 
let $\approx_A$ denote $\approx\upharpoonright (A\times A)$.
\end{definition}

Thus $\approx_{\mathsf{CDVS}}$ is the equivalence between countable distance value sets.

\begin{lemma}\label{lem:isomorphism} $\approx_{\mathsf{CDVS}}$ is Borel reducible to an isomorphism relation of countable structures. Consequently, $\approx_{\mathsf{CDVS}}$ is Borel reducible to the graph isomorphism.
\end{lemma}

\begin{proof} To each $\Delta\in \mathsf{CDVS}$ we associate a countable structure $\mathcal{A}_\Delta=(\Delta, R^{\mathcal{A}})$, where $R^{\mathcal{A}}$ is a ternary relation defined as
$$ R^{\mathcal{A}}(d_1,d_2,d_3)\iff \mbox{$(d_1,d_2,d_3)$ is a $\Delta$-triangle.}$$
It is clear that for $\Delta, \Lambda\in \mathsf{CDVS}$, we have $\Delta\approx \Lambda$ if and only if $\mathcal{A}_\Delta\cong \mathcal{A}_{\Lambda}$. This gives a Borel reduction from $\approx_{\mathsf{CDVS}}$ to an isomorphism relation of countable structures.
\end{proof}

Friedman--Stanley \cite[Theorem 3]{FS} showed that the graph isomorphism is Borel reducible to the isomorphism relation of all countable linear orders. Thus the following theorem determines the complexity of $\approx_{\mathsf{CDVS}}$ in the Borel reducibility hierarchy.

\begin{theorem}\label{thm:CDVS+} The isomorphism relation of all countable linear orders is Borel reducible to $\approx_{\mathsf{CDVS}}$. Consequently, the graph isomorphism is Borel reducible to $\approx_{\mathsf{CDVS}_+}$.
\end{theorem}

\begin{proof}
For a countable linear order $\mathcal{O}$, let $\psi(\mathcal{O})\subseteq (3,4)$  be an isomorphic copy of $\mathcal{O}$, and let 
$$\varphi(\mathcal{O})=\{2,4,6\}\cup\psi(\mathcal{O})\cup(2+\psi(\mathcal{O})). $$
Then $\varphi(\mathcal{O})\in \mathsf{CDVS}$. We verify that $\varphi$ is a Borel reduction from the isomorphism relation of all countable linear orders to $\approx_{\mathsf{CDVS}}$.

First suppose $\mathcal{O}_1$ and $\mathcal{O}_2$ are isomorphic countable linear orders. Then there is an order-preserving bijection $f$ between $\psi(\mathcal{O}_1)$ and $\psi(\mathcal{O}_2)$. For $x\in\varphi(\mathcal{O}_1)$, let 
$$F(x)=\left\{\begin{array}{ll} f(x),  & \mbox{ if $x\in\psi(\mathcal{O}_1)$,} \\
f(x-2)+2, & \mbox{ if $x\in 2+\psi(\mathcal{O}_1)$,} \\
x, & \mbox{ otherwise.}
\end{array}\right.
$$
It is straightforward to check that $F$ is a bijection between $\varphi(\mathcal{O}_1)$ and $\varphi(\mathcal{O}_2)$ witnessing their equivalence. Thus $\varphi(\mathcal{O}_1)\approx \varphi(\mathcal{O}_2)$.

Conversely, suppose $\varphi(\mathcal{O}_1)\approx \varphi(\mathcal{O}_2)$ and this equivalence is witnessed by a bijection $g$. Let $\Delta=\varphi(\mathcal{O}_1)$ and $\Lambda=\varphi(\mathcal{O}_2)$, and let $\mathcal{A}_\Delta=(\Delta, R^\Delta)$ and $\mathcal{A}_\Lambda=(\Lambda, R^\Lambda)$ be the countable structures defined in the proof of Lemma~\ref{lem:isomorphism}.  Then $g$ is an isomorphism witnessing $\mathcal{A}_\Delta\cong \mathcal{A}_\Lambda$. 

Now note that $2$ is definable in both $\mathcal{A}_\Delta$ and $\mathcal{A}_\Lambda$ by the formula $\exists y\neg R(x, x, y)$. Thus $g(2)=2$. 
Moreover, $\psi(\mathcal{O}_1)$ is definable in $\mathcal{A}_\Delta$ by the formula 
$$ x\neq 2\wedge R(2, 2, x) \wedge \exists y\neg R(2,x,y). $$
Similarly, $\psi(\mathcal{O}_2)$ is definable in $\mathcal{A}_\Lambda$ by the same formula. Thus $g$ sends $\psi(\mathcal{O}_1)$ to $\psi(\mathcal{O}_2)$. Finally, on $\psi(\mathcal{O}_1)$, the $x<y$ relation is definable by the formula
$$ \exists z\, [\neg R(2, x, z) \wedge R(2,y,z)]. $$
Similarly for $\psi(\mathcal{O}_2)$. Hence $g$ preserves the order $<$ as a bijection between $\psi(\mathcal{O}_1)$ and $\psi(\mathcal{O}_2)$. It follows that the linear orders $\mathcal{O}_1$ and $\mathcal{O}_2$ are isomorphic.

It is routine to check that $\varphi$ is Borel.
\end{proof}

Now Theorem~\ref{thm:main7} is essentially a direct consequence of Theorems~\ref{thm:class} and \ref{thm:CDVS+}. We only need to verify the Borelness of the reduction function.

     \begin{theorem}\label{reductions}
        The graph isomorphism is Borel reducible to $\cong_{\mathsf{EA}}$. In particular, $\cong_{\mathsf{EA}}$ is analytic and not Borel.
     \end{theorem}
     \begin{proof} In view of Theorem~\ref{thm:CDVS+}, we only need to verify $\approx_{\mathsf{CDVS}}\ \leq_B \ \cong_{\mathsf{EA}}$. This is witnessed by the map $\Delta\mapsto \mbox{\rm Aut}(\mathbb{U}^<_\Delta)$, and Theorem~\ref{thm:class} guarantees that it is a reduction. We only verify that this map is Borel. This is done in two steps. In the first step, we describe a Borel way to enumerate the class $\mathcal{K}^<_\Delta$, the class of all finite ordered $\Delta$-metric spaces, from a given sequence $(d_i)_i$ coding $\Delta$. For this, we fix a countable language $\mathcal{L}$ consisting of countably many binary relation symbols, i.e., $\mathcal{L}=\{R_i\}_{i\in \mathbb{N}} \cup \{<\}$. Enumerate those finite $\mathcal{L}$-structures $\mathcal{A}$ such that 
\begin{enumerate}
\item[(i)] for any distinct $x,y\in|\mathcal{A}|$, there is exactly one $i$ such that $d_i>0$ and $R^{\mathcal{A}}_i(x,y)$ holds;
\item[(ii)] for any $x,y,z\in|\mathcal{A}|$ and $i,j,k\in\mathbb{N}$, if $R^{\mathcal{A}}_i(x,y)$, $R^{\mathcal{A}}_j(y,z)$ and $R^{\mathcal{A}}_k(x,z)$, then $|d_j-d_k|\leq d_i\leq d_j+d_k$;
\item[(iii)] $<^{\mathcal{A}}$ is a linear order.
\end{enumerate}
It is clear that such structures correspond to finite ordered $\Delta$-metric spaces. 

Now, given an enumeration of $\mathcal{K}^<_\Delta$, one can construct the Fra\"iss\'e limit in a Borel way. This follows from the existence proof of Fra\"iss\'e's theorem (see,  \cite[Lemma~7.1.6]{Hodges}).  Since the Fra\"iss\'e structure is ultrahomogeneous, its automorphism group intersects a basic open set $N_{\bar{a},\bar{b}}$ if and only if the substructures $\bar{a}$ and $\bar{b}$ are isomorphic. Thus the map from the Fra\"iss\'e structures to their automorphism groups is Borel.  
\end{proof}

\bibliographystyle{amsplain}



\end{document}